\documentclass[a4paper,oneside,11pt, reqno]{amsproc}
\usepackage{graphicx}
\usepackage{float}
\usepackage{amscd}
\usepackage{amsmath}
\usepackage{amsxtra}
\usepackage[T1]{fontenc}
\usepackage[utf8]{inputenc}
\usepackage{lipsum}
\usepackage{tikz}
\usetikzlibrary{arrows}
\usetikzlibrary{calc}

\usepackage{amsfonts,amssymb,amsmath,amsthm}
\usepackage{url}
\usepackage{enumerate}

\newcommand{\strutstretchdef}{\newcommand{\strutstretch}{\vphantom{\raisebox{1pt}{$\big($}\raisebox{-1pt}{$\big($}}}}
\theoremstyle{plain}
\newtheorem{theorem}{Theorem}[section]

\newtheorem{lemma}[theorem]{Lemma}
\newtheorem{proposition}[theorem]{Proposition}
\newtheorem{corollary}[theorem]{Corollary}

\theoremstyle{definition}

\newtheorem{example}[theorem]{Example}
\newtheorem{question}[theorem]{Question}

\theoremstyle{remark}
\newtheorem{remark}[theorem]{Remark}

\numberwithin{equation}{section}

\newlength{\struh}
\newlength{\textminustop}
\strutstretchdef
\newcommand*{\Ge}{\geqslant}

\newcommand*{\Le}{\leqslant}

\usepackage{mathrsfs}
\usepackage{epsfig}
\usepackage[active]{srcltx}

\newcommand{\ncom}{\newcommand}
\ncom{\bq}{\begin{equation}}
\ncom{\eq}{\end{equation}}
\ncom{\beqn}{\begin{eqnarray*}}
\ncom{\eeqn}{\end{eqnarray*}}
\ncom{\beq}{\begin{eqnarray}}
\ncom{\eeq}{\end{eqnarray}}
\ncom{\nno}{\nonumber}
\ncom{\rar}{\rightarrow}
\ncom{\Rar}{\Rightarrow}
\ncom{\noin}{\noindent}
\ncom{\bc}{\begin{centre}}
\ncom{\ec}{\end{centre}}
\ncom{\sz}{\scriptsize}
\ncom{\rf}{\ref}
\ncom{\sgm}{\sigma}
\ncom{\Sgm}{\Sigma}
\ncom{\dt}{\delta}
\ncom{\Dt}{Delta}
\ncom{\lmd}{\lambda}
\ncom{\Lmd}{\Lambda}
\ncom{\eps}{\epsilon}
\ncom{\pcc}{\stackrel{P}{>}}
\ncom{\dist}{{\rm\,dist}}
\ncom{\sspan}{{\rm\,span}}
\ncom{\im}{{\rm Im\,}}
\ncom{\sgn}{{\rm sgn\,}}
\ncom{\ba}{\begin{array}}
\ncom{\ea}{\end{array}}
\ncom{\eop}{\hfill{{\rule{2.5mm}{2.5mm}}}}
\ncom{\eoe}{\hfill{{\rule{1.5mm}{1.5mm}}}}
\ncom{\eof}{\hfill{{\rule{1.5mm}{1.5mm}}}}
\ncom{\hone}{\mbox{\hspace{1em}}}
\ncom{\htwo}{\mbox{\hspace{2em}}}
\ncom{\hthree}{\mbox{\hspace{3em}}}
\ncom{\hfour}{\mbox{\hspace{4em}}}
\ncom{\hsev}{\mbox{\hspace{7em}}}
\ncom{\vone}{\vskip 2ex}
\ncom{\vtwo}{\vskip 4ex}
\ncom{\vonee}{\vskip 1.5ex}
\ncom{\vthree}{\vskip 6ex}
\ncom{\vfour}{\vspace*{8ex}}
\ncom{\norm}{\|\;\;\|}
\ncom{\integ}[4]{\int_{#1}^{#2}\,{#3}\,d{#4}}
\ncom{\inp}[2]{\langle{#1},\,{#2} \rangle}
\ncom{\Inp}[2]{\Langle{#1},\,{#2} \Langle}
\ncom{\vspan}[1]{{{\rm\,span}\#1 \}}}
\ncom{\dm}[1]{\displaystyle {#1}}

\begin{document}

\title[Bi-isometries reducing the hyper-ranges of the coordinates]{Bi-isometries reducing the hyper-ranges \\ of the coordinates}

\author[S. Chavan and Md. R. Reza]{Sameer Chavan and Md. Ramiz Reza}

\address{Department of Mathematics and Statistics\\
Indian Institute of Technology Kanpur, India}
   \email{chavan@iitk.ac.in}
\address{Indian Institute of Science Education and Research
Thiruvananthapuram, India}
 \email{ramiz@iisertvm.ac.in}



\begin{abstract} 
Let $(S_1, S_2)$ be a bi-isometry, that is, a pair of commuting isometries $S_1$ and $S_2$ on a complex Hilbert space $\mathscr H.$ By the von Neumann-Wold decomposition, the hyper-range $\mathscr H_\infty(S_1):=\cap_{n=0}^\infty S^n_1\mathscr H$ of $S_1$ reduces $S_1$ to a unitary operator. Although $\mathscr H_\infty(S_1)$ is an invariant subspace for $S_2,$ in general,
$\mathscr H_\infty(S_1)$ is not a reducing subspace for $S_2.$
We show that 
$\mathscr H_\infty(S_1)$ reduces $S_2$ to an isometry if and only if 
the subspaces $S_2(\ker S^*_1)$ and $\mathscr H_\infty(S_1)$ of $\mathscr H$ are orthogonal. 
Further, we describe all bi-isometries $(S_1, S_2)$ satisfying the orthogonality condition mentioned above.  
\end{abstract}

\maketitle

\section{Motivation and the main theorem}

For complex Hilbert spaces $\mathscr H$ and $\mathscr K,$ let $\mathscr B(\mathscr H, \mathscr K)$ denote the Banach space of bounded linear transformations from $\mathscr H$ into $\mathscr K.$ We denote $\mathscr B(\mathscr H, \mathscr H)$ simply by $\mathscr B(\mathscr H),$ the $C^*$-algebra of bounded linear operators on $\mathscr H.$ A closed subspace $\mathscr M$ of $\mathscr H$ is {\it invariant} for $T \in \mathscr B(\mathscr H)$ if $T(\mathscr M) \subseteq \mathscr M.$ We say that $\mathscr M$ {\it reduces} $T$ if $\mathscr M$ and $\mathscr H \ominus \mathscr M$ are invariant for $T.$ We say that $\mathscr M$ {\it essentially reduces} $T$ if $P_{\mathscr M}T|_{\mathscr H \ominus \mathscr M} \in \mathscr B(\mathscr H \ominus \mathscr M, \mathscr M)$ and $P_{\mathscr H \ominus \mathscr M}T|_{\mathscr M} \in \mathscr B(\mathscr M, \mathscr H \ominus \mathscr M)$ are compact linear transformations, where $P_{\mathscr K}$  denotes the orthogonal projection of $\mathscr H$ onto a closed subspace $\mathscr K$ of $\mathscr H.$ 
The {\it hyper-range} $\mathscr H_\infty(T)$ of $T \in \mathscr B(\mathscr H)$ is defined as the space $\cap_{n = 0}^\infty T^n\mathscr H.$
An operator $T \in \mathscr B(\mathscr H)$ is said to be a {\it contraction} (resp. an {\it isometry}) if $T^*T \Le I$ (resp. $T^*T=I$). An operator $T \in \mathscr B(\mathscr H)$ is {\it completely nonunitary} if there is no nonzero closed subspace $L$ of $\mathscr H$ that reduces $T$ to a unitary. 

A pair $(S_1, S_2)$ of operators $S_1, S_2 \in \mathscr B(\mathscr H)$  is said to be {\it commuting} if $S_1S_2=S_2S_1.$ A commuting pair $(S_1, S_2)$ is {\it doubly commuting} if  $S^*_1S_2=S_2S^*_1.$ A commuting pair $(S_1, S_2)$ of operators $S_1$ and $S_2$ in $\mathscr B(\mathscr H)$ is called a {\it bi-contraction on $\mathscr H$}  if $S_1$ and $S_2$ are contractions. 
A commuting pair $(S_1, S_2)$ of operators $S_1$ and $S_2$ in $\mathscr B(\mathscr H)$ is a {\it bi-isometry on $\mathscr H$}  if $S_1$ and $S_2$ are isometries.
If, in addition, $S_1$ and $S_2$ are unitaries, then we refer to $(S_1, S_2)$ as a {\it bi-unitary on $\mathscr H.$} Needless to say, there is considerable literature on bi-contractions and related classes (see, for example, \cite{AMY2021, BDF2012, BKS2013, GG2004, HM1989, KO1999, P2004, S1980}).

Let $\mathbb D$ denote the open unit disc centered at the origin in $\mathbb C.$ For a complex Hilbert space $\mathscr E,$ let $\mathscr H^2_{\mathscr E}(\mathbb D)$ denote the Hardy space of ${\mathscr E}$-valued holomorphic functions on $\mathbb D.$ For a complex Banach space $X,$ let $\mathscr H^\infty_X(\mathbb D)$ denote the Banach space of bounded $X$-valued holomorphic functions on $\mathbb D.$ For $\phi \in \mathscr H^\infty_{\mathscr B({\mathscr E})}(\mathbb D),$ let $M_{\phi, {\mathscr E}}$ denote the operator of multiplication by $\phi$ on $\mathscr H^2_{\mathscr E}(\mathbb D).$ 
If ${\mathscr E}=\mathbb C,$ we denote $\mathscr H^\infty_{\mathscr B({\mathscr E})}(\mathbb D)$ (resp. $M_{\phi, {\mathscr E}}$) simply by $\mathscr H^\infty(\mathbb D)$ (resp. $M_{\phi}$). The {\it $\mathscr B({\mathscr E})$-valued Schur class} of $\mathbb D$ is the closed unit ball of $\mathscr H^\infty_{\mathscr B({\mathscr E})}(\mathbb D).$
We say that $\phi \in \mathscr H^\infty_{\mathscr B({\mathscr E})}(\mathbb D)$ is an {\it inner function} if 
$\phi \in \mathscr H^\infty_{\mathscr B({\mathscr E})}(\mathbb D)$ and $\phi(\zeta)$ is an isometry for almost every $\zeta$ on the unit circle $\mathbb T$ (see \cite{AMY2021, SFBK2010} for the basic theory of the operator-valued Schur class).

The starting point of the present investigation is the following theorem that generalises the von Neumann-Wold decomposition (see \cite[Theorem~I.3.6]{Co1991} and \cite[Theorem~I.3.2]{SFBK2010}). 
\begin{theorem} \label{Wold-type-new}
Let $T \in \mathscr B(\mathscr H)$ be a contraction and let $\mathscr E:=\ker T^*.$ Consider the subspaces $\mathscr H_u(T)$ and $\mathscr H_s(T)$ of $\mathscr H$ given by
\beqn 
\mathscr H_u(T) &=& \bigcap_{n=1}^\infty \big\{h \in \mathscr H : \|T^nh\| = \|h\|=\|T^{*n}h\|\big\}, \\ \mathscr H_s(T) &=& \mathscr H \ominus \mathscr H_u(T).
\eeqn
Then $\mathscr H_u(T)$ reduces $T$ to the unitary $U=T|_{\mathscr H_u(T)},$ $S=T|_{\mathscr H_s(T)}$ is completely non-unitary and the decomposition 
$$T=U \oplus S~\mathrm{on}~\mathscr H= \mathscr H_u(T) \oplus \mathscr H_s(T)$$ is uniquely determined. 

In particular, if $T$ is an isometry, then 
$S$ is unitarily equivalent to the operator $M_{z, \mathscr E}$ of multiplication by $z$ on $\mathscr H^{2}_{\mathscr E}(\mathbb D),$ 
$\mathscr H_u(T)=\mathscr H_\infty(T)$ and 
\beqn
\mathscr H_s(T) = \bigvee_{w \in \mathbb D} \ker(S^*-w),  \quad
\mathscr H_s(T) =
\bigoplus_{n=0}^{\infty}S^n(\ker S^*).
\eeqn   
\end{theorem}
\begin{remark}
Note that $\mathscr H_u(T) \subseteq \mathscr H_\infty(T).$ In general, this inclusion is strict (for example, consider a  contraction having a non-unitary invertible orthogonal direct summand). 
\end{remark}

Let $T \in \mathscr B(\mathscr H)$ be a contraction. 
The parts $U = T|_{\mathscr H_u(T)}$ and $S = T|_{\mathscr H_s(T)}$ are called the {\it unitary part} and the {\it completely nonunitary
part} of $T$, respectively, and $T = U \oplus S$ is called the {\it canonical decomposition} of $T.$  
We say that $T$ has {\it Wold-type decomposition} if 
\beq \label{span-new}
\mathscr H_u(T)=\mathscr H_\infty(T)~\mbox{and}~
\mathscr H_s(T) = \bigvee_{n=0}^{\infty}S^n(\ker S^*)
 \eeq
(cf.   \cite[Definition~1.1]{S2001}).
\begin{remark} 
Let $T \in \mathscr B(\mathscr H)$ be a left invertible operator. If $T$ satisfies $TT^* + (T^*T)^{-1} \Le 2 I,$ 
then  
$T$ has the Wold-type decomposition. This may be deduced from  \cite[Theorem~3.6]{S2001}.  
\end{remark}

The present note is motivated by the following question concerning the structure of bi-isometries.

\begin{question} \label{Q1}
Let $(S_1, S_2)$ be a bi-isometry (resp. bi-contraction) on $\mathscr H.$ When $\mathscr H_\infty(S_1)$ reduces $S_2$ to an isometry (resp. a contraction)?  
\end{question}

In case $S_2$ has finite dimensional cokernel, it has been shown in \cite[Theorem~2.3]{BKS2013} that $\mathscr H_\infty(S_1)$ reduces $S_2$ to an isometry. In general, $\mathscr H_\infty(S_1)$ is not a reducing subspace for $S_2.$  
\begin{example}
Let $\varphi : \mathbb D \rar \mathbb C$ belong to the Schur class of the open unit disc $\mathbb D.$  
Let $\mu$ denote the weighted arc-length measure $w\frac{d\theta}{2\pi}$ with weight function $w=1-|\varphi|^2.$  Let $\mathscr G:=L^2(\mathbb T, \mu)$ and let $V : \mathscr G \rar \mathscr G$ be the operator of multiplication by the coordinate function $\zeta$ in $\mathscr B(\mathscr G).$ For an integer $n \Ge 0,$ let $e_n \in \mathscr G$ be given by $e_n(\zeta)=\zeta^n,$ $\zeta \in \mathbb T.$
Define $B : \mathscr H^2(\mathbb D) \rar \mathscr H^2_{\mathscr G}(\mathbb D)$ by
\beqn
B(z^n) = e_n, \quad n \Ge 0,
\eeqn  
which is extended linearly and continuously as a bounded linear operator to the whole of $\mathscr H^2(\mathbb D)$ (indeed, $\|B\| \Le \mu(\mathbb T)$).  
Note that $$M_VB(z^n) = V(e_n)=e_{n+1}=B(z^{n+1})=BM_z(z^n), \quad n \Ge 0,$$ and hence we obtain  
\beq \label{MvBMz}
M_VB = BM_z.
\eeq
Define bounded linear operators $S_1, S_2$ on $\mathscr H^2_{\mathscr G}(\mathbb D) \oplus \mathscr H^2(\mathbb D)$ by
\beqn
S_1 = \Big[\begin{smallmatrix} M_V & 0 \\
0 & M_z \end{smallmatrix}\Big], ~ S_2 = \Big[\begin{smallmatrix} M_{z, {\mathscr G}} & B \\
0 & M_\varphi \end{smallmatrix}\Big]~\mbox{on}~\mathscr H^2_{\mathscr G}(\mathbb D) \oplus \mathscr H^2(\mathbb D).
\eeqn
Clearly, $S_1$ is an isometry. We check that $S_2$ is also an isometry.
Since $M^*_{z, {\mathscr G}}B(z^n)=M^*_{z, {\mathscr G}}(e_n)=0$ (as $e_n \in {\mathscr G}$), we have $M^*_{z, {\mathscr G}}B =0,$ and hence 
\beq \label{S2-iso-new}
S^*_2S_2 = \Big[\begin{smallmatrix} I & M^*_{z, {\mathscr G}}B \\
B^*M_{z, {\mathscr G}} & B^*B + M^*_\varphi M_\varphi \end{smallmatrix}\Big] =  \Big[\begin{smallmatrix} I & 0 \\
0 & B^*B + M^*_\varphi M_\varphi \end{smallmatrix}\Big].
\eeq
Moreover, for any $f, g \in \mathscr H^2(\mathbb D),$
\beqn
\inp{(B^*B + M^*_\varphi M_\varphi)f}{g} &=& \inp{Bf}{Bg} + \inp{\varphi f}{\varphi g} \\
&=& \int_{\mathbb T}f(\zeta)\overline{g(\zeta)}\,d\mu(\zeta)  \\
&+& \int_{\mathbb T}|\varphi(\zeta)|^2f(\zeta)\overline{g(\zeta)}\, d\zeta,
\eeqn
which is equal to $\inp{f}{g}$ since $d\mu(\zeta) = (1-|\varphi(\zeta)|^2)d\zeta.$ It now follows from \eqref{S2-iso-new} that $S_2$ is an isometry. 
Moreover, 
\beqn
S_1S_2 = 
\Big[\begin{smallmatrix} M_VM_{z, {\mathscr G}} & M_VB \\
0 & M_zM_\varphi \end{smallmatrix}\Big], \quad
S_2S_1 = 
\Big[\begin{smallmatrix} M_{z, {\mathscr G}}M_V & BM_z \\
0 & M_\varphi M_z \end{smallmatrix}\Big].
\eeqn
This, combined with \eqref{MvBMz}, shows that 
$S_1$ and $S_2$ commute.
Thus $(S_1, S_2)$ is a bi-isometry. 
Note that 
\beqn
\mbox{$\varphi$ is an inner function} \, \Longleftrightarrow \, \mu =0 \, \Longleftrightarrow \, \mu(\mathbb T)=0\, \Longleftrightarrow \, {\mathscr G}=\{0\}.
\eeqn 
Since $B(1) = \mu(\mathbb T)$ and $\mathscr H_\infty(S_1) = \mathscr H^2_{\mathscr G}(\mathbb D),$   
\beqn
\mbox{$\mathscr H_\infty(S_1)
$ is a reducing subspace for $S_2$} & \Longleftrightarrow & B=0\\ &\Longleftrightarrow & \mbox{  
$\varphi$ is an inner function.}
\eeqn
Assume that $\varphi$ is not inner. Thus 
$\mathscr H^2(\mathbb D) \subseteq {\mathscr G},$ so that ${\mathscr G}$ is infinite dimensional. Moreover, since $B(1) \neq 0,$   
\beqn
S_2(\ker S^*_1)= \{\big[\begin{smallmatrix} \alpha  \\
\alpha \varphi \end{smallmatrix}\big] : \alpha \in \mathbb C\}.
\eeqn In particular, $S_2(\ker S^*_1)$ and $\mathscr H_\infty(S_1)$ are not orthogonal. 
\eof
\end{example}

If $S_1$ has the Wold-type decomposition, then it turns out that the orthogonality of $S_2(\ker S^*_1)$ and $\mathscr H_\infty(S_1)$ is the {\it only} decisive factor in answer to Question~\ref{Q1} (see Remark~\ref{rmk-gen}).

\begin{theorem} \label{main-thm}
Let $(S_1, S_2)$ be a bi-isometry on $\mathscr H$ and let ${\mathscr E}:=\ker S^*_1.$ 
The following statements are equivalent$:$
\begin{enumerate}
\item[$\mathrm{(i)}$] 
$\mathscr H_\infty(S_1)$ reduces $S_2$ to an isometry,
\item[$\mathrm{(ii)}$]
$\mathscr H_\infty(S_1)$ reduces $(S_1, S_2)$ to a doubly commuting bi-isometry, 
\item[$\mathrm{(iii)}$] $S_2({\mathscr E})$ and $\mathscr H_\infty(S_1)$ are orthogonal subspaces of $\mathscr H,$
\item[$\mathrm{(iv)}$] $P_\infty \bigvee \{S^{*k}_1S_2(x) : k \Ge 1\}$ is finite-dimensional for every $x \in {\mathscr E},$ where $P_\infty$ denotes the orthogonal projection of $\mathscr H$ onto $\mathscr H_\infty(S_1),$
\item[$\mathrm{(v)}$] 
$\mathscr H_\infty(S_1)$ essentially reduces $S_2$ to an isometry.
\end{enumerate}
If $\mathrm{(i)}$ holds, then there exist an inner function $\phi : \mathbb D \rar \mathscr B({\mathscr E}),$ a Hilbert space ${\mathscr F},$ a constant unitary operator-valued function $\psi : \mathbb D \rar \mathscr B({\mathscr F}),$ a closed subspace $\mathscr H_{uu}$ of $\mathscr H_\infty(S_1)$ and  
a bi-unitary $(V_{1}, V_{2})$ on $\mathscr H_{uu}$
such that $(S_1, S_2)$ is unitarily equivalent to
\beq
\label{main-decom}
 (V_{1}, V_{2}) \oplus (M_{\psi, {\mathscr F}}, M_{z, {\mathscr F}}) \oplus (M_{z, {\mathscr E}}, M_{\phi, {\mathscr E}}) \, \mbox{on} \, \mathscr H_{uu} \oplus \mathscr H^2_{{\mathscr F}}(\mathbb D) \oplus \mathscr H^2_{{\mathscr E}}(\mathbb D).
\eeq 
\end{theorem}
\begin{remark} \label{rmk-gen}
An examination of the proof of Theorem~\ref{main-thm}, as presented in Section~2,  shows that if $(S_1, S_2)$ be a bi-contraction on $\mathscr H$ and $S_1$ has the Wold-type decomposition, then $\mathscr H_\infty(S_1)$ reduces $S_2$ to a contraction if and only if $S_2({\mathscr E})$ and $\mathscr H_\infty(S_1)$ are orthogonal subspaces of $\mathscr H.$
\end{remark}

\section{Proof of Theorem~\ref{main-thm}}

Let $(S_1, S_2)$ be a bi-contraction on $\mathscr H.$ 
Throughout this section, we use the following notations$:$
\begin{enumerate}
\item[$\bullet$] ${\mathscr E}=\ker S^*_1$
\item[$\bullet$] $P_\infty$ the orthogonal projection of $\mathscr H$ onto $\mathscr H_\infty(S_1)$
\end{enumerate}

The proof of Theorem~\ref{main-thm}, as presented below, requires several lemmas. 
\begin{lemma} \label{no-eigen}
Let $(S_1, S_2)$ be a bi-contraction on $\mathscr H$. 
Consider the closed subspace $\mathscr M= \bigvee_{\lambda \in \mathbb C} \ker(S_1-\lambda)$ of $\mathscr H.$ 
Then the following are valid$:$
\begin{enumerate}
\item[$\mathrm{(i)}$] for every $\lambda \in \mathbb C,$ $\ker(S_1-\lambda)$ reduces $(S_1, S_2)$ and 
\beq \label{def-M-equi}
\mathscr M= \bigoplus_{\lambda \in \mathbb T} \ker(S_1-\lambda).
\eeq 
\item[$\mathrm{(ii)}$] $\mathscr M$ is a subspace of $\mathscr H_u(S_1)$ such that 
\beqn
\mathscr H_u(S_1) = \mathscr M\oplus \mathscr H_u({S_1}|_{\mathscr H \ominus \mathscr M}).
\eeqn
\end{enumerate}
\end{lemma}
\begin{proof}  
(i) 
 Since $S_1|_{\mathscr H_u(S_1)}$ is unitary (see Theorem	\ref{Wold-type-new}) and $\ker(S_1-\mu),$ $\mu \in \mathbb T,$ is a subspace of $\mathscr H_u(S_1),$ the spaces $\ker(S_1-\lambda)$ and $\ker(S_1-\mu)$ are orthogonal for any $\lambda \in \mathbb T$ provided $\lambda \neq \mu.$
Since the eigenvalues of any isometry on a Hilbert space lie on the unit circle, we obtain \eqref{def-M-equi}.
Let $\lambda \in \mathbb T.$ 
Clearly, $\ker(S_1-\lambda)$ is an invariant subspace of $S_1.$
Also, for any $x \in \ker(S_1-\lambda),$ 
$S_1(S_2x)=\lambda S_2x$ (since $S_1S_2=S_2S_1$), and hence 
$S_2x \in \ker(S_1-\lambda).$ Thus 
$\ker(S_1-\lambda)$ is invariant for $S_2.$ 
To see that $\ker(S_1-\lambda)$ reduces the pair $(S_1, S_2),$
let $x \in \ker(S_1-\lambda).$ By 
 \cite[Proposition 1.3.1]{SFBK2010}, $S^*_1x = \overline{\lambda}x,$ and hence $S^*_1x \in \ker(S_1-\lambda)$ showing that $\ker(S_1-\lambda)$ reduces $S_1.$ This also shows that
 $S^*_2(S^*_1x - \overline{\lambda}x)=0$ or equivalently, $S^*_1(S^*_2x)= \overline{\lambda}S^*_2x.$  Another application of  \cite[Proposition 1.3.1]{SFBK2010} shows that $S_1(S^*_2x)= \lambda S^*_2x$ or $S^*_2x \in \ker(S_1-\lambda)$ completing the verification of (i).

(ii) Clearly, $\mathscr M$ is a subspace of $\mathscr H_u(S_1)$ such that $S_1\mathscr M=\mathscr M.$ Thus $\mathscr H_u({S_1}|_{\mathscr M})=\mathscr M,$ and hence by (i), $\mathscr H_u(S_1) = \mathscr M\oplus \mathscr H_u({S_1}|_{\mathscr H \ominus \mathscr M}).$ 
\end{proof}
 
A particular case of part (i) below (when $(S_1, S_2)$ is a bi-isometry) has been obtained in \cite[Lemma~1]{S1980} by an entirely different method. Also, the space $\mathscr N_x,$ as given in (ii) below, appears in the proof of \cite[Theorem~2.3]{BKS2013}.
\begin{lemma} \label{main-lem}
Let $(S_1, S_2)$ be a bi-contraction on $\mathscr H.$ Assume that  $S_1$ is a left invertible operator that has the Wold-type decomposition $S_1 = U \oplus S$ on $\mathscr H=\mathscr H_u(S_1) \oplus \mathscr H_s(S_1).$ 
Then the following statements are valid$:$
\begin{enumerate}
\item[$\mathrm{(i)}$]   
$\mathscr H_u(S_1)$ is an invariant subspace for $S_2,$
\item[$\mathrm{(ii)}$]  if 
$\mathscr N_x:=P_\infty \bigvee \{S^{*k}_1S_2(x) : k \Ge 1\}$ is finite-dimensional for every $x \in {\mathscr E},$
then $\mathscr H_u(S_1)$ is a reducing subspace for $S_2,$
\item[$\mathrm{(iii)}$] if $\mathscr H_u(S_1)$ reduces $S_2,$  then $({S_1}|_{\mathscr H_u(S_1)}, {S_2}|_{\mathscr H_u(S_1)})$ is a doubly commuting bi-contraction. 
\end{enumerate}
\end{lemma}
\begin{proof}  
Write $S_2 = \big[\begin{smallmatrix} A & B \\
C & D \end{smallmatrix}\big]$ relative to the orthogonal decomposition $\mathscr H = \mathscr H_u(S_1) \oplus \mathscr H_s(S_1).$ 
Since $S_2$ commutes with $S_1= \big[\begin{smallmatrix} U & 0 \\
0 & S \end{smallmatrix}\big],$ a routine calculation shows that 
\beq \label{c-relations}
UA = AU, ~CU=SC, ~UB=BS, ~SD=DS.
\eeq

(i) 
We adopt an argument from \cite[Theorem~2.1]{KS2000} to the present situation.  
It follows from \eqref{c-relations} that for any $f \in \ker(S^*-w),$ $w \in \mathbb D,$
\beqn
{w}C^*f = C^*S^* f = U^*C^*f,
\eeqn
and hence $(U^*-{w}I)C^*f=0.$ Since $\sigma(U^*)$ is a subset of the unit circle, $C^*f=0$ for any $f \in \ker(S^*-w)$ and every $w \in \mathbb D.$ 
Thus it suffices to check that
\beq \label{wsp-new}
\mathscr H_s(S_1) =\bigvee_{w \in \mathbb D} \ker(S^*-w).
\eeq
Since $S=S_1|_{\mathscr H_s(S_1)}$ is an analytic left invertible operator, by the Shimorin's analytic model (see \cite[p.~154]{S2001}),  $S$ is unitarily equivalent to the operator $M_z$ of multiplication by $z$ on a reproducing kernel Hilbert space $\mathscr H_\kappa$ for some reproducing kernel $\kappa : \mathbb D_r \times \mathbb D_r \rar \mathscr B(\mathscr E),$ where $\mathbb D_r$ denote the disc centred at the origin and of radius $r >0.$  In view of $$
\mathscr H_\kappa =\bigvee \{\kappa(\cdot, w) : w \in \mathbb D_r\}=\bigvee_{w \in \mathbb D_r} \ker(M^*_z-w),$$
we obtain \eqref{wsp-new}.  

(ii)  
Because of Lemma~\ref{no-eigen}, after replacing $\mathscr  H$ by $\mathscr H \ominus \mathscr M,$ if required, we may assume without loss of generality that 
\beq \label{no-e-v}
\mbox{$S_1$ has no eigenvalues.}
\eeq
Assume that $\dim \mathscr N_x < \infty$ for every $x \in {\mathscr E}.$ 
We claim that $B=0.$ Since $S_1$ has Wold-type decomposition (see \eqref{span-new}), 
$\mathscr H_s(S_1) = \bigvee_{n=0}^{\infty}S^n(\ker S^*).$ Thus, 
it suffices to check that $BS^n|_{\ker S^*}=0$ for every integer $n \Ge 0.$ However, by \eqref{c-relations}, $BS^n=U^nB,$ and hence it is enough to check that $B|_{\ker S^*}=0.$ Let $x \in \ker S^*.$ Since 
$\mathscr H_u(S_1)=\mathscr H_\infty(S_1),$  
$U={S_1}|_{\mathscr H_u(S_1)}$ commutes with $P_\infty,$ and hence 
$U^{*k}  = P_\infty {S^{*k}_1}|_{\mathscr H_u(S_1)},$ $k \Ge 1.$ Since $S_2 = \big[\begin{smallmatrix} A & B \\
0 & D \end{smallmatrix}\big],$ it follows that  
\beqn
\mathscr N_x = \bigvee \{U^{*k}Bx : k \Ge 1\}.
\eeqn
Since $\mathscr N_x$ is finite dimensional, $U^*\mathscr N_x \subseteq \mathscr N_x$ and $U$ is unitary, the operators $U^*$ and (equivalently) $U$ have an eigenvalue provided $\mathscr N_x \neq \{0\}.$ In view of \eqref{no-e-v}, this is not possible forcing that $\mathscr N_x=\{0\}.$ It follows that $Bx=0$ completing the verification of (ii). 

(iii) Since  ${S_2}|_{\mathscr H_u(S_1)}$ commutes with the unitary operator ${S_1}|_{\mathscr H_u(S_1)},$ this is immediate.  
\end{proof}

To prove Theorem~\ref{main-thm}, we need another lemma, which describes all doubly commuting bi-isometries (see Corollary~\ref{Sl-thm} below for a finer version). 
\begin{lemma} \label{d-commuting-iso}
Let $(S_1, S_2)$ be a bi-isometry on $\mathscr H.$  
If the pair $(S_1, S_2)$ is doubly commuting, then there exist an isometry $V \in \mathscr B(\mathscr H_\infty(S_1))$ and a constant inner function $\phi : \mathbb D \rar \mathscr B({\mathscr E})$ such that 
\beq
\label{d-commuting-eqn}
(S_1, S_2)  \cong  (U, V) \oplus  (M_{z, {\mathscr E}}, M_{\phi, {\mathscr E}}) ~ \mbox{on} ~ \mathscr H_\infty(S_1) \oplus \mathscr H^2_{{\mathscr E}}(\mathbb D),
\eeq
where $U={S_1}|_{\mathscr H_\infty(S_1)}$ and 
$\cong$ denotes the unitary equivalence.
If, in addition, $S_2$ is unitary, then $\phi$ is a constant unitary operator-valued function. 
\end{lemma}
\begin{proof} 
Note that $\mathscr H_u(S_1)=\mathscr H_\infty(S_1)$ (see Theorem~\ref{Wold-type-new}). 
Assume that $(S_1, S_2)$ is doubly commuting. 
Thus, $S^{*k}_1S_2(x)=0$ for every $x \in {\mathscr E},$ 
and hence by Lemma~\ref{main-lem}(ii), 
$\mathscr H_\infty(S_1)$ is a reducing subspace for $S_2$ and $V:={S_2}|_{\mathscr H_\infty(S_1)}$ is an isometry. By \cite[Lemma~V.3.2]{SFBK2010} and \cite[Proposition~V.2.2]{SFBK2010}, there exists an inner function $\phi \in H^{\infty}_{\mathscr B({\mathscr E})}(\mathbb D)$ such that ${S_2}|_{H \ominus \mathscr H_\infty(S_1)} \cong M_{\phi, {\mathscr E}}.$ 
Thus \eqref{d-commuting-eqn} holds provided we check that $\phi$ is constant.  
Since $(S_1, S_2)$ is doubly commuting,  
\beq \label{m-phi-m-z}
M^*_{\phi, {\mathscr E}}M_{z, {\mathscr E}} = M_{z, {\mathscr E}}M^*_{\phi, {\mathscr E}}. 
\eeq
Let $P_0$ denote the orthogonal projection of $\mathscr H^2_{{\mathscr E}}(\mathbb D)$ onto the space ${\mathscr E}$ of constant polynomials. 
Note that $M_{z, {\mathscr E}}M^*_{z, {\mathscr E}}=I-P_0,$ and hence 
\beq \label{Toep-id}
M^*_{\phi, {\mathscr E}}(I-P_0) \overset{\eqref{m-phi-m-z}} =M_{z, {\mathscr E}}M^*_{\phi, {\mathscr E}}M^*_{z, {\mathscr E}}.
\eeq
By \cite[Exercise~2.66]{AM2002},  
\beq \label{text-ex}
M^*_{\phi, {\mathscr E}}\kappa(\cdot, w)f = \kappa(\cdot, w) \phi(w)^*f,  \quad f \in {\mathscr E}, ~w \in \mathbb D,
\eeq
where $\kappa(z, w)=\frac{I_{{\mathscr E}}}{1-z\overline{w}},$ $z, w \in \mathbb D$ with $I_{{\mathscr E}}$ denoting the identity operator on $\mathscr E.$ 
Thus for any $f, g \in {\mathscr E}$ and $u, v \in \mathbb D,$
\beqn
&& \inp{M^*_{\phi, {\mathscr E}}(I-P_0)\kappa(\cdot, v)f}{\kappa(\cdot, u)g} \\
& \overset{\eqref{text-ex}}  = & \inp{\kappa(\cdot, v)\phi(v)^*f}{\kappa(\cdot, u)g} - \inp{\phi(0)^*f}{\kappa(\cdot, u)g},
\eeqn
and also
\beqn
 \inp{M_{z, {\mathscr E}}M^*_{\phi, {\mathscr E}}M^*_{z, {\mathscr E}} \kappa(\cdot, v)f}{\kappa(\cdot, u)g} 
 \overset{\eqref{text-ex}}  =  
u\overline{v} \inp{\kappa(\cdot, v)\phi(v)^*f}{\kappa(\cdot, u)g}.
\eeqn
This combined with the reproducing property $$\inp{h}{\kappa(\cdot, u)f}=\inp{h(u)}{f}, \quad h \in \mathscr H^2_{{\mathscr E}}(\mathbb D), ~ f \in {\mathscr E},$$ 
and \eqref{Toep-id} yields $\phi(v)=\phi(0)$ for every $v \in \mathbb D.$  To see the remaining part, one may argue as above to show that 
\beqn
M_{\phi, {\mathscr E}} M^*_{\phi, {\mathscr E}} = I ~\Rightarrow \phi(0)\phi(0)^*=I,
\eeqn
which completes the proof. 
\end{proof}

\begin{proof}[Proof of Theorem~\ref{main-thm}]  By Theorem~\ref{Wold-type-new}, $S_1$ has the Wold-type decomposition $U \oplus S$ on $\mathscr H=\mathscr H_\infty(S_1) \oplus \mathscr H_s(S_1).$ 
By Lemma~\ref{main-lem}(i),  
$S_2$ decomposes as $\big[\begin{smallmatrix} A & B \\
0 & D \end{smallmatrix}\big]$ on $\mathscr H = \mathscr H_u(S_1) \oplus \mathscr H_s(S_1).$ 
In view of the proof of Lemma~\ref{main-lem}, we may assume that \eqref{c-relations} holds. Clearly, (i) holds if and only if $B=0.$ 

(i)$\Rightarrow$(ii) This follows from Lemma~\ref{main-lem}(iii).

(ii)$\Rightarrow$(iii)
Since $B=P_\infty S_2|_{\mathscr H_s(S_1)}$ and ${\mathscr E} \subseteq \mathscr H_s(S_1),$ (ii) implies that $P_\infty S_2({\mathscr E})=\{0\},$ and hence (iii) holds. 

(iii)$\Rightarrow$(iv) By (iii), $P_\infty S_2({\mathscr E})=\{0\}.$ Since $P_\infty$ commutes with $S_1,$ for every $x \in {\mathscr E},$
$P_\infty \bigvee \{S^{*k}_1S_2(x) : k \Ge 1\}=\{0\}.$ 

(iv)$\Rightarrow$(i) This follows from Lemma~\ref{main-lem}(ii).

(i)$\Leftrightarrow$(v) Clearly, it is enough to verify that (v)$\Rightarrow$(i). In turn, it suffices to check that if $B=P_\infty S_2|_{\mathscr H_u(S_1)}$ is compact, then $B=0.$ 
By \eqref{c-relations}, $UB=BS,$ and hence
\beqn
(UB)^*UB = (BS)^*BS ~ \Rightarrow ~ B^*B = S^*B^*B S.
\eeqn
Thus $B^*B$ is $S$-Toeplitz operator in the sense of \cite{DE2011}. Now, if $B$ is compact, then $B^*B$ is a compact $S$-Toeplitz operator, and since $S$ does not have eigenvalues, by \cite[Theorem~3.3]{DE2011}, $B^*B=0$ or equivalently, $B=0.$ 

Assume now that (i) holds. Then  
\beq \label{T1-Wold}
S_2 = A \oplus D ~\mbox{on} ~\mathscr H= \mathscr H_\infty(S_1) \oplus {\mathscr H}_s(S_1),
\eeq
where 
$A$ and $D$ are contractions such that 
\beq \label{c-relations-new}
AU=UA, ~ DS=SD.
\eeq 
This yields 
\beq \label{deco-new-bi-c} 
(S_1, S_2) = (U, A) \oplus (S, D)
~\mbox{on}~\mathscr H = \mathscr H_\infty(S_1) \oplus \mathscr H_{s}(S_1). \eeq
An application of \cite[Lemma~V.3.2]{SFBK2010} together with \eqref{c-relations-new} shows that there exists $\phi \in \mathscr H^{\infty}_{\mathscr B({\mathscr E})}(\mathbb D)$ such that $D =M_{\phi, {\mathscr E}}.$ 
Since $D$ is an isometry, by \cite[Proposition~V.2.2]{SFBK2010}, $\phi$ is an inner function. 
It now follows from \eqref{deco-new-bi-c} and Theorem~\ref{Wold-type-new} that 
\beqn (S_1, S_2) \cong (U, A) \oplus (M_z, M_\phi)
~\mbox{on}~\mathscr H = \mathscr H_\infty(S_1) \oplus \mathscr H^2_{{\mathscr E}}(\mathbb D). \eeqn
Moreover, by Lemma~\ref{main-lem}(iii), $(U, A)$ is a doubly commuting bi-isometry. Since $U$ is unitary,  one may now apply Lemma \ref{d-commuting-iso} (applied to $S_1=A$ and $S_2=U$) to obtain \eqref{main-decom}. 
\end{proof}

We conclude this section with two consequences of Theorem~\ref{main-thm}. The first one recovers 
\cite[Theorem~2]{S1980}.  
\begin{corollary}[M. Słociński] \label{Sl-thm}
Let $(S_1, S_2)$ be a doubly commuting bi-isometry on $\mathscr H.$ 
There exist Hilbert spaces ${\mathscr E_j},$ $j=1, \ldots, 3,$  constant unitary operator-valued functions $\phi : \mathbb D \rar \mathscr B({\mathscr E_1})$ and $\psi : \mathbb D \rar \mathscr B({\mathscr E_2}),$ a closed subspace $\mathscr H_{uu}$ of $\mathscr H_\infty(S_1)$ and a bi-unitary $(V_{1}, V_{2})$ on $\mathscr H_{uu}$
such that 
\beqn
(S_1, S_2) \cong  (V_{1},  V_{2}) \oplus (M_{\phi, {\mathscr E_1}}, M_{z, {\mathscr E_1}}) \oplus (M_{z, {\mathscr E_2}}, M_{\psi, {\mathscr E_2}}) \oplus  (M_{z, {\mathscr E_3}}, M_{z, {\mathscr E_3}})\\
 ~\mbox{on~}\mathscr H_{uu} \oplus \mathscr H^2_{{\mathscr E_1}}(\mathbb D) \oplus \mathscr H^2_{{\mathscr E_2}}(\mathbb D) \oplus \mathscr H^2_{{\mathscr E_3}}(\mathbb D). 
\eeqn
\end{corollary}
\begin{proof} Since $(S_1, S_2)$ is a doubly commuting bi-isometry, condition (iv) of Theorem~\ref{main-thm} is satisfied. Hence, by Theorem~\ref{main-thm}, 
there exist Hilbert space ${\mathscr F},$  an inner function $\phi : \mathbb D \rar \mathscr B({\mathscr E}),$ a constant unitary operator-valued function $\psi : \mathbb D \rar \mathscr B({\mathscr F}),$ a closed subspace $\mathscr H_{uu}$ of $\mathscr H_\infty(S_1)$ and a bi-unitary $(V_{1}, V_{2})$ on $\mathscr H_{uu}$
such that $(S_1, S_2)$ admits the decomposition \eqref{main-decom}. Apply now Lemma~\ref{d-commuting-iso} to the doubly commuting bi-isometry $(M_{\phi, {\mathscr E}}, M_{z, {\mathscr E}})$ to get the desired decomposition. 
\end{proof}

The following result provides some finiteness conditions ensuring the orthogonal decomposition \eqref{main-decom}. In particular, it generalises \cite[Theorem~2.3]{BKS2013}.
\begin{corollary} \label{coro-main}
Let $(S_1, S_2)$ be a bi-isometry on $\mathscr H.$ 
Assume that any one of the following conditions holds$:$
\begin{enumerate}
\item[$\mathrm{(a)}$] $\dim \, \ker(P_\infty {S^*_2}|_{\mathscr H_\infty(S_1)}) < \infty,$
\item[$\mathrm{(b)}$] $\dim \, P_\infty \ker(S^*_2) < \infty,$
\item[$\mathrm{(c)}$] $\mathrm{card}\, \sigma({S_1}|_{\mathscr H_\infty(S_1)}) < \infty,$ where $\mathrm{card}$ denotes the cardinality. 
\end{enumerate}   
Then there exist a Hilbert space ${\mathscr F},$ an inner function $\phi : \mathbb D \rar \mathscr B({\mathscr E}),$ a constant unitary operator-valued function $\psi : \mathbb D \rar \mathscr B({\mathscr F}),$ a closed subspace $\mathscr H_{u}$ of $\mathscr H_\infty(S_1)$ and a bi-unitary $(V_{1}, V_{2})$ on $\mathscr H_{u}$
such that $(S_1, S_2)$ admits the decomposition 
\eqref{main-decom}.
\end{corollary}
\begin{proof}  
By Lemma~\ref{main-lem}(i), $S_2=\big[\begin{smallmatrix} A & B \\
0 & D \end{smallmatrix}\big],$ where $A, B, D$ satisfy \eqref{c-relations}. After replacing $\mathscr H$ by $\mathscr H \ominus \mathscr M,$ if required, we may assume that \eqref{no-e-v} holds. 

Assume that (a) holds.
A routine calculation using $S^*_2S_2=I$ shows that 
\beq
\label{A-iso}
A^*A =I,  \quad A^*B=0.
\eeq
By \eqref{c-relations}, $UA=AU.$ Since $U$ is unitary, $UA^*=A^*U.$ Thus, $\ker A^*$ is a reducing subspace for $U.$ If $\ker A^* \neq \{0\},$ then $U$ has an eigenvalue (since by (a), $\ker A^*$ is finite dimensional), and hence $S_1$ has eigenvalue, which is not possible in view of \eqref{no-e-v}. Thus $\ker A^* = \{0\},$ and hence by \eqref{A-iso}, $B=0.$

To see the remaining assertions, 
for $x \in \mathscr H,$ let 
\beq 
\label{Mx}
\mathscr N_x :=\bigvee \{U^{*k}P_\infty S_2(x) : k \Ge 1\}. 
\eeq
In view of Lemma~\ref{main-lem}(ii), it suffices to check that 
$\mathscr N_x$  is finite-dimensional for any $x \in {\mathscr E}$ provided either of (b) and (c) holds.

First, assume that (b) holds. 
Since $S_1S_2=S_2S_1$ and $S^*_2S_2=I,$ for any positive integer $k,$ 
\beqn
S^*_2\Big(S^{*k}_1S_2\big[\begin{smallmatrix} 0 \\
x \end{smallmatrix}\big]\Big) =S^{*k}_1\big[\begin{smallmatrix} 0 \\
x \end{smallmatrix}\big] = \big[\begin{smallmatrix} 0 \\
S^{*k}x \end{smallmatrix}\big] =  0.
\eeqn
This combined with  
\beqn
S^{*k}_1S_2\big[\begin{smallmatrix} 0 \\
x \end{smallmatrix}\big] =\big[\begin{smallmatrix} U^{*k} & 0 \\
0 & S^{*k} \end{smallmatrix}\big]\big[\begin{smallmatrix} Bx \\
Dx \end{smallmatrix}\big] = \big[\begin{smallmatrix} U^{*k}Bx \\
S^{*k}Dx \end{smallmatrix}\big]
\eeqn
shows that $\big[\begin{smallmatrix} U^{*k}Bx \\
S^{*k}Dx \end{smallmatrix}\big] \in \ker S^*_2$ for every integer $k \Ge 1.$
Thus, by \eqref{Mx}, $\mathscr N_x \subseteq P_\infty (\ker S^*_2),$ and hence
$\dim \mathscr N_x$ is finite for every $x \in {\mathscr E}.$ 

Next, assume that (c) holds. Thus there exists a polynomial $p$ vanishing on $\sigma(U),$ where $U:={S_1}|_{\mathscr H_\infty(S_1)}$ is a unitary operator.  
By the spectral mapping property, $\sigma(p(U))=\{0\}.$ Since $p(U)$ is a normal operator, $p(U)=0.$ It is now easy to see that $\dim \mathscr N_x$ is finite for every $x \in {\mathscr E}.$ 
\end{proof}

It is worth noting that the conclusion of Corollary~\ref{coro-main} extends naturally to commuting $n$-tuples $(S_1, \ldots, S_n)$ of isometries for which the dimension of $\ker \prod_{j=1}^n S^*_j$ is finite. The proof of this fact can be given along the lines of proof of \cite[Theorem 2.4]{BKS2013}. We leave the details to the reader. 

\section{Bi-isometries not reducing hyper-ranges of the coordinates}

In this short section, we discuss a class of bi-isometries $(S_1, S_2)$ on $\mathscr H$ for which $\mathscr H_\infty(S_1)$ is not a reducing subspace of $S_2.$
We begin with the following consequence of the proof of Theorem~\ref{main-thm} (see Lemma~\ref{main-lem}(i), \eqref{c-relations} and the discussion following \eqref{T1-Wold}).
\begin{proposition} \label{prop-coro-main}
If $(S_1, S_2)$ is a bi-isometry on $\mathscr H,$ then there exist a unitary operator $U \in \mathscr B(\mathscr H_\infty(S_1)),$ $\phi \in \mathscr H^{\infty}_{\mathscr B({\mathscr E})}(\mathbb D)$ with $\|\phi\|_{\infty} \Le 1,$ an isometry $A \in \mathscr B(\mathscr H_\infty(S_1))$ and $B \in \mathscr B(\mathscr H^2_{\mathscr E}(\mathbb D), \mathscr H_\infty(S_1))$ such that 
\beqn
\mbox{$UA = AU,$ $UB=BM_{z, {\mathscr E}},$ $A^*B=0,$ $B^*B+ M^*_{\phi, {\mathscr E}} M_{\phi, {\mathscr E}} =I$ and} \\
(S_1, S_2) \cong \Big(\Big[\begin{smallmatrix} U & 0 \\
0 & M_{z, {\mathscr E}} \end{smallmatrix}\Big],  \Big[\begin{smallmatrix} A & B \\
0 & M_{\phi, {\mathscr E}} \end{smallmatrix}\Big]\Big)~\mbox{on}~\mathscr H_\infty(S_1) \oplus \mathscr H^2_{\mathscr E}(\mathbb D),
\eeqn 
where $\cong$ denotes the unitary equivalence. 
\end{proposition}
The previous proposition raises the following question$:$
\begin{question} \label{Q2}
With the notations as in Proposition~\ref{prop-coro-main}, when does there exist a nonzero $B \in \mathscr B(\mathscr H^2_{\mathscr E}(\mathbb D), \mathscr H_\infty(S_1))$ such that $UB=BM_{z, {\mathscr E}},$ $A^*B=0$ and $B^*B+M^*_{\phi, {\mathscr E}} M_{\phi, {\mathscr E}} =I$?
\end{question}

Here is an answer to Question~\ref{Q2} in case of ${\mathscr E}=\mathbb C.$
\begin{proposition} \label{answer-Q2}
Under the notations of Proposition~	\ref{prop-coro-main} with ${\mathscr E}=\mathbb C,$   
\beq \label{BP-B1}
UB=BM_{z} ~\Longleftrightarrow~B(p) = p(U)B(1)~\mbox{for every polynomial}~p.
\eeq
If $UB=BM_z,$ then 
the following statements are valid$:$
\begin{enumerate}
\item[$\mathrm{(i)}$]  $A^*B=0$ if and only if $A^*B(1)=0,$
\item[$\mathrm{(ii)}$]
if $\mu$ denotes the spectral measure of the unitary operator $U$ $($extended trivially to $\mathbb T),$ then $B^*B+M^*_\phi M_\phi =I$ if and only if 
\beqn 
d\inp{\mu(\cdot)B(1)}{B(1)} = (1-|\phi|^2)d\theta.
\eeqn
\end{enumerate}
\end{proposition}
\begin{proof} 
Since the polynomials are dense in $\mathscr H^2(\mathbb D),$ 
\eqref{BP-B1} is immediate. 
To see (i) and (ii), assume that $UB=BM_z.$

(i) Assume that $A^*B(1)=0.$ By \eqref{BP-B1}, $A^*B(p)=A^*p(U)B(1)$ for every polynomial $p.$ However, since $AU=UA,$ by  
$A^*B(p)=p(U)A^*B(1)=0.$ Since polynomials are dense in $\mathscr H^2(\mathbb D),$ $A^*B=0.$ 
The converse is trivial. 

(ii) Let 
$p, q$ be polynomials in $z.$ By \eqref{BP-B1},
\beqn
&& \inp{B(p)}{B(q)} + \inp{M_\phi(p)}{M_\phi(q)} \\
&=& \inp{p(U)B(1)}{q(U)B(1)} + \inp{\phi p}{\phi q} \\
&=& \int_{\mathbb T} p(e^{i \theta})\overline{q(e^{i \theta})} d\inp{\mu(e^{i \theta})B(1)}{B(1)} 
+ \int_{\mathbb T}p(e^{i \theta})\overline{q(e^{i \theta})}|\phi(e^{i \theta})|^2d \theta.
\eeqn
Thus $\inp{B(p)}{B(q)} + \inp{M_\phi(p)}{M_\phi(q)}= \inp{p}{q}$ if and only if 
\beqn
\int_{\mathbb T} p(e^{i \theta})\overline{q(e^{i \theta})} d\inp{\mu(e^{i \theta})B(1)}{B(1)} = \int_{\mathbb T}
p(e^{i \theta})\overline{q(e^{i \theta})}(1-|\phi(e^{i \theta})|^2)d \theta.
\eeqn 
The desired conclusion in (ii) now follows from the uniqueness of the trigonometric moment problem (see 
\cite[Theorem~1.4]{ST1943}). 
\end{proof}
 
\begin{corollary}
Under the notations of Proposition~	\ref{prop-coro-main} with ${\mathscr E}=\mathbb C,$  if the arc-length measure of $\sigma(U)$ is $0,$ then $\mathscr H_\infty(S_1)$ reduces $(S_1, S_2)$ to a doubly commuting bi-isometry on $\mathscr H_\infty(S_1),$ 
\end{corollary}
\begin{proof}
By Proposition~\ref{prop-coro-main}, 
there exists $\phi \in \mathscr H^{\infty}(\mathbb D)$ with $\|\phi\|_{\infty} \Le 1$ such that $B^*B+ M^*_{\phi} M_{\phi} =I.$ It now follows from Proposition~\ref{answer-Q2}(ii) and the assumption that the arc-length measure of $\sigma(U)$ is $0$ that $\phi$ is an inner function. Thus $M_\phi$ is an isometry or, equivalently, $B=0.$ One may now apply Theorem~\ref{main-thm}. 
\end{proof}
It would be interesting to know if the conclusion of the previous corollary holds in case ${\mathscr E}$ is of a dimension bigger than $1.$

{}

\end{document}